\documentclass[sts]{imsart}

\usepackage{amssymb,amsbsy,amsmath,amscd,amsthm,amsfonts,MnSymbol}
\usepackage[utf8]{inputenc}
\usepackage{enumerate}
\usepackage{hyperref}
\usepackage{dsfont}
\usepackage{graphicx}
\usepackage{enumitem}
\usepackage{wrapfig}

\usepackage{algorithm}
\usepackage{algorithmic}

\newtheorem{definition}{Definition}
\newtheorem{theorem}{Theorem}

\newtheorem{lemma}{Lemma}
\newtheorem{remark}{Remark}

\newtheorem{corollary}{Corollary}
\newtheorem{proposition}{Proposition}

\newtheorem{fact}{Fact}

\newcommand{\R}{\mathbb R}
\newcommand{\C}{\mathbb C}
\newcommand{\PP}{\mathbb P}
\newcommand{\E}{\mathbb E}

\newcommand{\HH}{\mathcal H}
\newcommand{\TT}{\mathcal T}
\newcommand{\SN}{\mathcal S}

\newcommand{\GF}{\textsf{GF}_2}
\newcommand{\DPP}{\textsf{DPP}}
\newcommand{\Sig}{\mathfrak{S}}
\newcommand{\Signs}{\{-1,+1\}}

\newcommand{\cov}{\textsf{cov}}

\newcommand{\DS}{\displaystyle}

\begin{document}

\begin{frontmatter}

\title{Learning Signed Determinantal Point Processes through the Principal Minor Assignment Problem}
\runtitle{Signed DPPs}


\author{\fnms{Victor-Emmanuel} \snm{Brunel}\ead[label=e1]{vebrunel@mit.edu}}
\address{\printead{e1}}
\affiliation{Massachusetts Institute of Technology, Department of Mathematics}

\runauthor{V.-E. Brunel}

\begin{abstract}
Symmetric determinantal point processes (DPP's) are a class of probabilistic models that encode the random selection of items that exhibit a repulsive behavior. They have attracted a lot of attention in machine learning, when returning diverse sets of items is sought for. Sampling and learning these symmetric DPP's is pretty well understood. In this work, we consider a new class of DPP's, which we call signed DPP's, where we break the symmetry and allow attractive behaviors. We set the ground for learning signed DPP's through a method of moments, by solving the so called principal assignment problem for a class of matrices $K$ that satisfy $K_{i,j}=\pm K_{j,i}$, $i\neq j$, in polynomial time.

\end{abstract}



\end{frontmatter}

\section{Introduction}

Random point processes on finite spaces are probabilistic distributions that allow to model random selections of sets of items from a finite collection. For example, the basket of a random customer in a store is a random subset of items selected from that store. In some contexts, random point processes are encoded as random binary vectors, where the $1$ coordinates correspond to the selected items. A very famous subclass of random point processes, much used in statistical mechanics, is called the Ising model, where the log-likelihood function is a quadratic polynomial in the coordinates of the binary vector. More generally, \textit{Markov random fields} encompass models of random binary vectors where stochastic dependence between the coordinates of the random vector is encoded in an undirected graph. In recent years, a different family of random point processes has attracted a lot of attention, mainly for its computational tractability: \textit{determinantal point processes} (DPP's). DPP's were first studied and used in statistical mechanics \cite{Mac75}. Then, following the seminal work \cite{KulTas11}, discrete DPP's have been used increasingly in various applications such as recommender systems \cite{GarPaqKoe16,GarPaqKoe16b}, document and timeline summarization \cite{LinBil12, YaoFanZha16}, image search \cite{KulTas11,AffFoxAda14} and segmentation \cite{LeeChaYan16}, audio signal processing \cite{XuOu16}, bioinformatics \cite{BatQuoKul14} and neuroscience \cite{SnoZemAda13}.

A DPP on a finite space is a random subset of that space whose inclusion probabilities are determined by the principal minors of a given matrix. More precisely, encode the finite space with labels $[N]=\{1,2,\ldots,N\}$, where $N$ is the size of the space. A DPP is a random subset $Y\subseteq [N]$ such that $\PP[J\subseteq Y]=\det(K_J)$, for all fixed $J\subseteq [N]$, where $K$ is an $N\times N$ matrix with real entries, called the \textit{kernel} of the DPP, and $K_J=(K_{i,j})_{i,j\in J}$ is the square submatrix of $K$ associated with the set $J$. In the applications cited above, it is assumed that $K$ is a symmetric matrix. In that case, it is shown (e.g., see \cite{KulTas12}) that a sufficient and necessary condition for $K$ to be the kernel of a DPP is that all its eigenvalues are between $0$ and $1$. In addition, if $1$ is not an eigenvalue of $K$, then the DPP with kernel $K$ is also known as an \textit{$L$-ensemble}, where the probability mass function is proportional to the principal minors of the matrix $L=K(I-K)^{-1}$, where $I$ is the $N\times N$ identity matrix. DPP's with symmetric kernels, which we refer to as \textit{symmetric DPP's}, model repulsive interactions: Indeed, they imply a strong negative dependence between items, called \textit{negative association}  \cite{BorBraLig09}. 

Recently, symmetric DPP's have become popular in recommender systems, e.g., automatized systems that seek for \textit{good} recommendations for users on online shopping websites \cite{GarPaqKoe16}. The main idea is to model a random basket as a DPP and learn the kernel $K$ based on previous observations. Then, for a new customer, predict which items are the most likely to be selected next, given his/her current basket, by maximizing the conditional probability $\PP[J\cup\{i\}\subseteq Y|J\subseteq Y]$ over all items $i$ that are not already in the current basket $J$. One very attractive feature of DPP's is that the latter conditional probability is tractable and can be computed in a polynomial time in $N$. However, if the kernel $K$ is symmetric, this procedure enforces diversity in the baskets that are modeled, because of the negative association property. Yet, in general, not all items should be modeled as repelling each other. For instance, say, on a website that sells household goods, grounded coffee and coffee filters should rather be modeled as attracting each other, since a user who buys grounded coffee is more likely to also buy coffee filters. In this work, we extend the class of symmetric DPP's in order to account for possible attractive interactions, by considering nonsymmetric kernels. In the learning prospective, this extended model poses a question: How to estimate the kernel, based on past observations? For symmetric kernels, this problem has been tackled in several works \cite{GilKulFox14,AffFoxAda14,MarSra15,BarTit15,DupBac16,GarPaqKoe16,GarPaqKoe16b,MarSra16,BruMoiRigUrs2017,BruMoiRigUrs2017bis}. Here, we assume that $K$ is nonparametric, in the sense that it is not parametrized by a low dimensional parameter. As explained in \cite{BruMoiRigUrs2017} in the symmetric case, the maximum likelihood approach requires to solve a highly non convex optimization problem, and even though some algorithms have been proposed such as fixed point algorithms \cite{MarSra16}, Expectation-Maximisation \cite{GilKulFox14}, MCMC \cite{AffFoxAda14}, no statistical guarantees are given for these algorithms. The method of moments proposed in \cite{BruMoiRigUrs2017bis} provides a polynomial time algorithm based on the estimation of a small number of principal minors of $K$, and finding a symmetric matrix $\hat K$ whose principal minors approximately match the estimated ones. This algorithm is closely related to the \textit{principal minor assignment problem}. Here, we are interested in learning a nonsymmetric kernel given available estimates of its principal minors; In order to simplify the exposition, we always assume that the available list of principal minors is exact, not approximate.

In Section \ref{Sec:Def}, we recall the definition of DPP's together with general properties, we characterize the set of admissible kernels under lack of symmetry and we define a new class of nonsymmetric kernels, that we call \textit{signed kernels}. We tackle the question of identifiability of the kernel of a signed DPP and show that this question, together with the problem of learning the kernel, is related to the \textit{principal minor assignment problem}. In Section \ref{Sec:Sol}, we propose a solution to the principal minor assignment problem for signed kernels, which yields a polynomial time learning algorithm for the kernel of a signed DPP.

\section{Determinantal Point Processes} \label{Sec:Def}

\subsection{Definitions}

\begin{definition}[Discrete Determinantal Point Process]
A Determinantal Point Process (\textit{DPP}) on the finite set $[N]$ is a random subset $Y\subseteq [N]$ for which there exists a matrix $K\in\R^{N\times N}$ such that the following holds:
\begin{equation} \label{DefDPP}
	\PP[J\subseteq Y]=\det(K_J), \hspace{4mm} \forall J\subseteq [N],
\end{equation}
where $K_J$ is the submatrix of $K$ obtained by keeping the columns and rows of $K$ whose indices are in $J$. The matrix $K$ is called the \textit{kernel} of the DPP, and we write $Y\sim\textsf{DPP}(K)$.
\end{definition}

In short, the inclusion probabilities of a DPP are given by the principal minors of some matrix $K$. As we will see below, $K$ is not uniquely determined for a given DPP, even though, for simplicity, we say ``\textit{the} kernel" instead of ``\textit{a} kernel".

\begin{definition}[$L$-ensembles]
An $L$-ensemble on the finite set $[N]$ is a random subset $Y\subseteq [N]$ for which there exists a matrix $L\in\R^{N\times N}$ such that the following holds:
\begin{equation} \label{DefLEns}
	\PP[Y=J]\propto \det(L_J), \quad \forall J\subseteq [N].
\end{equation}
\end{definition}
In this definition, the symbol $\propto$ means an equality up to a multiplicative constant that does not depend on $J$. By simple linear algebra, it is easy to see that the multiplicative constant must be $\det(I+L)^{-1}$. Similarly as above, $L$ is not uniquely determined for a given $L$-ensemble.

\begin{proposition} \label{PropKtoL}
	An $L$-ensemble is a DPP, with kernel $K=L(I+L)^{-1}$, where $L$ is defined in \eqref{DefLEns}. Conversely, a DPP with kernel $K\in\R^{N\times N}$ is an $L$-ensemble if and only if $I-K$ is invertible. In that case, the matrix $L$ is given by $L=K(I-K)^{-1}$.
\end{proposition}

The proof of this proposition follows the lines of \cite[Theorem 2.2]{KulTas12}, which actually does not use the symmetry of $K$ and $L$. Even when $I-K$ is not invertible, the probability mass function of $\DPP(K)$ has a simple closed form formula. For all $J\subseteq [N]$, we denote by $\mathds 1_J$ the $N\times N$ diagonal matrix whose $j$-th diagonal entry is $1$ if $j\in J$, $0$ otherwise, and denote by $\bar J$ the complement of the set of $J$ in $[N]$.

\begin{lemma} \label{LemmaIE}
	Let $Y\sim\DPP(K)$, for some $K\in\R^{N\times N}$. Then, 
	\begin{equation*}
		\PP[Y=J]=(-1)^{|\bar J|}\det(K-\mathds 1_{\bar J})=|\det(K-\mathds 1_{\bar J})|, \quad \forall J\subseteq [N].
	\end{equation*}
\end{lemma}

\begin{proof}
	Let $J\subseteq [N]$. Then, by the inclusion-exclusion principle,
\begin{align}
	\PP[Y=J] & = \PP[J\subseteq Y, J\cup\{i\}\not\subseteq Y, \forall i\notin J] \nonumber \\
	& = \sum_{S\subseteq \bar J} (-1)^{|S|}\PP[J\cup S\not\subseteq Y] \nonumber \\
	& = \sum_{S\subseteq \bar J} (-1)^{|S|}\det(K_{J\cup S}) \nonumber \\
	& = (-1)^{|\bar J|}\det(K-\mathds 1_{\bar J}),
\end{align}
where the last inequality is a consequence of the multilinearity of the determinant. 
\end{proof}

\subsection{Admissibility of a kernel} \label{sec:Admiss}

Note that not all matrices $K\in\R^{N\times N}$ give rise to a DPP since, for instance, the numbers $\det(K_J)$ from \eqref{DefDPP} must all lie in $[0,1]$, and be nonincreasing with the set $J$. We call a matrix $K\in\R^{N\times N}$ \textit{admissible} if there exists a DPP with kernel $K$. When $K$ is symmetric, it is well known that it is admissible if and only if all its eigenvalues are between $0$ and $1$ (see, e.g., \cite{KulTas12}).
As a straightforward consequence of Lemma \ref{LemmaIE}, we have the following proposition. 

\begin{proposition}
	For all matrices $K\in\R^{N\times N}$, $K$ is admissible if and only if $(-1)^{|J|}\det(K-\mathds 1_J)\geq 0$, for all $J\subseteq [N]$. 
\end{proposition}

\begin{proof}
	Let $K$ be admissible and let $Y\sim\DPP(K)$. By Lemma \ref{LemmaIE}, $0\leq \PP[Y=J]=(-1)^{N-|J|}\det(K-\mathds 1_{\bar J})$ for all $J\subseteq [N]$. Conversely, assume $(-1)^{|J|}\det(K-\mathds 1_J)\geq 0$ for all $J\subseteq [N]$. Denote by $p_J=(-1)^{|\bar J|}\det(K-\mathds 1_{\bar J})$, for all $J\subseteq [N]$. By a standard computation, $\DS \sum_{J\subseteq [N]}p_J=1$, hence, one can define a random subset $Y\subseteq [N]$ with $\PP[Y=J]=p_J$ for all $J\subseteq [N]$. The same application of the inclusion-exclusion principle as in the proof of Lemma \ref{LemmaIE} yields that $\PP[J\subseteq Y]=\det(K_J)$ for all $J\subseteq [N]$, hence, $Y\sim\textsf{DPP}(K)$.
\end{proof}

Let $K\in\R^{N\times N}$. Assume that $I-K$ is invertible. Then, by Proposition \ref{PropKtoL}, $K$ is admissible if and only if the matrix $L=K(I-K)^{-1}$ defines an $L$-ensemble. In that case, Lemma \ref{LemmaIE} yields that $\det(L_J)/\det(I+L)=(-1)^{|\bar J|}\det(K-\mathds 1_{\bar J})$, for all $J\subseteq [N]$. Hence, $K$ is admissible if and only if $L$ is a $P_0$-matrix, i.e., all its principal minors are nonnegative. If, in addition, $K$ is invertible, then it is admissible if and only if $L$ is a $P$-matrix, i.e., all its principal minors are positive, if and only if $TK+(I-T)(I-K)$ is invertible for all diagonal matrices $T$ with entries in $[0,1]$ (see \cite[Theorem 3.3]{johnson1995convex}). In particular, it is easy to see that any matrix $K$ of the form $D+\mu A$, where $D$ is a diagonal matrix with $D_{i,i}\in [\lambda,1-\lambda], i=1,\ldots,N$, for some $\lambda\in (0,1/2)$, $A\in [-1,1]^{N\times N}$ and $0\leq \mu< \lambda/(N-1)$, is admissible.

\begin{remark}
	For symmetric kernels, admissibility is equivalent for all the eigenvalues to lie in $[0,1]$. In general, the (complex) eigenvalues of an admissible kernel need not even lie in the band $\{z\in\C: 0\leq \mathcal R(z)\leq 1\}$. For instance, if both $K$ and $I-K$ are invertible, admissibility of $K$ is equivalent to $L=K(I-K)^{-1}$ being a $P$-matrix (see Section \ref{sec:Admiss}). By \cite[Theorem 2.5.9.a]{hom1991topics}, the eigenvalues of a $P$-matrix $L$ can be anywhere in the wedge $\{z=re^{i\theta}: r>0, |\theta|<\pi(1-1/N)\}$. Since the eigenvalues of $K$ are given by $z/(1+z)$, for all eigenvalues $z$ of $L$, we conclude that an admissible kernel $K$ can have eigenvalues arbitrarily far away from the band $\{z\in\C: 0\leq \mathcal R(z)\leq 1\}$.
\end{remark}

\subsection{General properties}

DPP's are known to be stable under simple operations such as marginalization, conditioning: We review some of these properties, which can also be found in \cite{KulTas12} in the case of symmetric kernels, or in \cite{Bor11} for $L$-ensembles. 

\begin{proposition} \label{PropStable}
Let $K$ be an admissible kernel and $Y\sim\DPP(K)$.

\begin{itemize}

\item Marginalization: For all $S\subseteq [N]$, $K_S\in \R^{S\times S}$ is admissible and $Y\cap S$ is a DPP with kernel $K_S$;

\item Complement: $I-K$ is admissible and $\bar Y=[N]\setminus Y\sim \DPP(I-K)$;

\item Conditioning: Let $S\subseteq [N]$ such that $\det(K_S)\neq 0$. Then, $K+(I-K)\mathds 1_{\bar S}$ is invertible, $\DS I_{\bar S}-\left[\left(K+(I-K)\mathds 1_{\bar S}\right)^{-1}(I-K)\right]_{\bar S}\in\R^{\bar S\times \bar S}$ is admissible and conditionally on the event $S\subseteq Y$, $\DS Y\cap \bar S\sim\DPP\left(I_{\bar S}-\left[\left(K+(I-K)\mathds 1_{\bar S}\right)^{-1}(I-K)\right]_{\bar S}\right)$.

\end{itemize}

\end{proposition}

\begin{proof}
\begin{itemize}
	\item Marginalization: This property is straightforward, after noticing that for all $J\subseteq S$, $\DS \PP[J\subseteq Y\cap S]=\PP[J\subseteq Y]=\det(K_J)=\det\left((K_S)_J\right)$.

	\item Complement: That $I-K$ is admissible follows from the fact that for all $J\subseteq [N]$, 
	\begin{align*}
		(-1)^{|J|}\det\left((I-K)-\mathds 1_{J}\right) & = (-1)^{|J|}\det\left(\mathds 1_{\bar J}-K\right) \\
		& = (-1)^{|\bar J|}\det\left(K-\mathds 1_{\bar J}\right) \\
		& \geq 0,
	\end{align*}
	by admissibility of $K$. Then, for all $J\subseteq [N]$,
	\begin{align*}
		\PP[\bar Y=J] & = \PP[Y=\bar J] \\
		& = (-1)^{|J|}\det(K-\mathds 1_J) \\
		& = (-1)^{|\bar J|}\det\left((I-K)-\mathds 1_{\bar J}\right).
	\end{align*}
	\item Conditioning: Note that $K+(I-K)\mathds 1_{\bar S}=\mathds 1_{\bar S}+K\mathds 1_S$, hence each column $j\in \bar S$ has $N-1$ zeros and one one, on the diagonal, so column reduction yields $\det(K+(I-K)\mathds 1_{\bar S})=\det(K_S)\neq 0$, hence, $K+(I-K)\mathds 1_{\bar S}$ is invertible.

	Now, for all $J\subseteq \bar S$, by Bayes' rule, plus the fact that for all $T\subseteq [N]$, $\det(K_T)=\det(K+(I-K)\mathds 1_{\bar T})$, and finally by a column reduction of the determinant,
	\begin{align*}
		\PP[J\subseteq Y\cap \bar S | S\subseteq Y] & = \frac{\PP[S\cup J\subseteq Y]}{\PP[S\subseteq Y]} \\
		& = \frac{\det(K+(I-K)\mathds 1_{\overline{S\cup J}})}{\det(K+(I-K)\mathds 1_{\bar S})} \\ 
		& = \det\left((K+(I-K)\mathds 1_{\bar S})^{-1}(K+(I-K)\mathds 1_{\overline{S\cup J}})\right) \\
		& = \det\left(I-(K+(I-K)\mathds 1_{\bar S})^{-1}(I-K)\mathds 1_{J})\right) \\
		& = \det\left(I_{J}-\left[\left(K+(I-K)\mathds 1_{\bar S}\right)^{-1}(I-K)\right]_{J}\right) \\
		& = \det(\tilde K_J),
	\end{align*}
where $\DS \tilde K=I_{\bar S}-\left[\left(K+(I-K)\mathds 1_{\bar S}\right)^{-1}(I-K)\right]_{\bar S} \in\R^{\bar S\times \bar S}$.
\end{itemize}
\end{proof}

For a DPP with a symmetric kernel $K$, it is known that the eigenstructure of $K$ plays a role, e.g., in sampling \cite[Section 2.4.4]{KulTas12}. For a general DPP, with non necessarily symmetric kernel $K$, the eigenstructure of $K$ does not seem to play a significant role, either in learning or sampling. Indeed, the eigenvalues of $K$ are complex numbers and $K$ may not be diagonalizable. However, we show that the eigenvalues of $K$, even if they may be non real, completely characterize the distribution of the size of the DPP. In the sequel, we denote by $\mathcal R(z)$ (resp. $\mathcal I(z)$) the real part (resp. imaginary part) of the complex number $z$.
\begin{lemma} \label{LemmaSupMat}
	Let $K\in\R^{N\times N}$ be an admissible kernel and let $Y\sim\DPP(K)$. Let $\lambda_1,\ldots,\lambda_p$ be the real eigenvalues of $K$, repeated according to their multiplicity and let $\mu_1,\ldots,\mu_q$ be the eigenvalues of $K$ that have positive imaginary part, also accounting for their multiplicity. Then, for all complex numbers $z$, 
\begin{equation*}
	\E[z^{|Y|}] = \prod_{j=1}^p(1-\lambda_j+z\lambda_j) \prod_{k=1}^q (1+2\mathcal R(\mu_k)(z-1) +(z-1)^2|\mu_k|^2).
\end{equation*}
\end{lemma}

Using the same notation as in the lemma, we note that, since $K$ is a real matrix, its eigenvalues are exactly $\lambda_1,\ldots,\lambda_p, \mu_1,\overline{\mu_1},\ldots,\mu_q,\overline{\mu_q}$ (repeated according to their multiplicity). In particular, $p+2q=N$.

\begin{proof}

Assume first that $I-K$ is invertible, so that $Y$ is an $L$-ensemble and $\DS \PP[Y=J]=\frac{\det(L_J)}{\det(I+L)}$, for all $J\subseteq {N}$, with $L=K(I-K)^{-1}$. Then, for all $z\in\C$, 
\begin{align}
	\E[z^{|Y|}] & = \sum_{J\subseteq [N]}\frac{\det(L_J)}{\det(I+L)}z^{|J|} \nonumber \\
	& = \sum_{J\subseteq [N]}\frac{\det((zL)_J)}{\det(I+L)} \nonumber \\
	& = \frac{\det(I+zL)}{\det(I+L)} \nonumber \\
	& = \det(I+zK(I-K)^{-1})\det(I-K) \nonumber \\
	& = \det(I-K+zK) \label{stepMGF} \\
	& = \prod_{j=1}^p(1-\lambda_j+z\lambda_j) \prod_{k=1}^q (1-\mu_k+z\mu_k)(1-\overline{\mu_k}+z\overline{\mu_k}) \nonumber \\
	& = \prod_{j=1}^p(1-\lambda_j+z\lambda_j) \prod_{k=1}^q (1+2\mathcal R(\mu_k)(z-1) +(z-1)^2|\mu_k|^2). \nonumber 
\end{align}
The conclusion of the lemma follows by extending this computation to the case when $I-K$ is not invertible, by continuity.
\end{proof}

In particular, we have the following corollary.

\begin{corollary} 
With all the same notation as in Lemma \ref{LemmaSupMat}, if all the non real eigenvalues of $K$ lie in the complex disk with center $1/2$ and radius $1/2$. Then $|Y|$ has the same distribution as $U_1+\ldots+U_p+V_1+V_2+\ldots+V_{2q-1}+V_{2q}$, where:

\begin{itemize}
	\item[-] $U_j\sim\textsf{Ber}(\lambda_j)$, for all $j\in [p]$;
	\item[-] $V_{2k-1}$ and $V_{2k}$ are $\textsf{Ber}(\mathcal R(\mu_k))$, for all $k\in [q]$
	\item[-] $\textsf{cov}(V_{2k-1},V_{2k})=(\mathcal I(\mu_k))^2$, for all $k\in [q]$;
	\item[-] The random variables $U_1,\ldots,U_p$ and the pairs $(V_1,V_2),\ldots,(V_{2q-1},V_{2q})$ are all mutually independent.
\end{itemize}

\end{corollary}

For example, let $K=D+\mu A$, where $D$ is a real diagonal matrix with $D_{i,i}\in [\lambda,1-\lambda]$ for all $i\in [N]$, for some $\lambda\in (0,1/2)$, $\mu\geq 0$ and $A\in[-1,1]^{N\times N}$. If $\mu<\lambda/(N-1)$, then $K$ is admissible (see above) and, by Gerschgorin's circle theorem, all the eigenvalues of $K$ lie in one of the complex disks with center $D_{i,i}$ and radius $\mu$, $i=1,\ldots,N$, hence, in the complex disk with center $1/2$ and radius $1/2$.

\begin{proof}

First, recall that by Proposition \ref{PropStable}, all principal submatrices of both $K$ and $I-K$ are admissible, yielding that $K$ and $I-K$ are $P_0$-matrices. By \cite[Theorem 2.5.6]{hom1991topics}, all the real eigenvalues of a $P$-matrix are nonnegative. Since for all $\varepsilon>0$, $K+\varepsilon I$ (resp. $I-K+\varepsilon I$) is a $P$-matrix, its real eigenvalues are all nonnegative; Its real eigenvalues are exactly the $\lambda_j+\varepsilon$ (resp. $1-\lambda_j+\varepsilon$), $j=1,\ldots,p$; By taking the limit as $\varepsilon$ goes to zero, all real eigenvalues of $K$ (resp. $I-K$) are nonnegative, hence, $0\leq \lambda_j\leq 1, \forall j=1,\ldots,p$.

Moreover, note that a complex number $\mu$ lies in the disk with center $1/2$ and radius $1/2$ if and only if $\mathcal R(\mu)\geq |\mu|^2$. Hence, for all $k\in [q]$, the polynomial (in $z$) $1+2\mathcal R(\mu_k)(z-1)+(z-1)^2|\mu_k|^2$ has real and nonnegative coefficients; So, by Lemma \ref{LemmaSupMat}, the moment generating function of $|Y|$ is the moment generating function of the sum of $p+q$ independent random variables, namely, $U_1,\ldots,U_p,(V_1+V_2),\ldots,(V_{2q-1}+V_{2q})$.
\end{proof}

It is easy to see that if $Y\sim\DPP(K)$ for some admissible kernel $K$, then $\DS \textsf{Var}(|Y|)=\textsf{Tr}\left(K^\top(I-K)\right)$. If $K$ is symmetric, this yields $\DS \textsf{Var}(|Y|)=\sum_{j=1}^N \lambda_j(1-\lambda_j)$, where $\lambda_1,\ldots,\lambda_N$ are the eigenvalues of $K$. It is well known that all the eigenvalues of a symmetric admissible kernel are in $[0,1]$ (see \cite[Section 2.1]{KulTas12}). Therefore, if $K$ is symmetric, then $|Y|$ has constant size if and only if $0$ and $1$ are the only eigenvalues of $K$, i.e., $K$ is an orthogonal projection. The following corollary shows that this still holds true for general DPP's, except that even if $0$ and $1$ are the only eigenvalues of an admissible kernel, it does not have to be a projection matrix in general.

\begin{corollary}
	Let $K$ be an admissible kernel and $Y\sim\DPP(K)$. Then, $Y$ has almost surely constant size if and only if $0$ and $1$ are the only eigenvalues of $K$.
\end{corollary}

\begin{proof}
	Assume that $Y$ has constant size, i.e., $|Y|=p$ almost surely, for some $p\in \{0,1,\ldots,N\}$. Then, $\E[z^{|Y|}]=z^p$ for all $z\in\C$. Let $\mu$ be an eigenvalue of $K$. If $\mu$ is not real, then by Lemma \ref{LemmaSupMat}, the polynomial $1+2\mathcal R(\mu)(z-1) +(z-1)^2|\mu|^2$ must divide $z^p$, hence it must be a monomial, i.e., of the form $C(\mu)z^k$ for some $C(\mu)\in\R$ and $k\in \{0,1,2\}$. Note that $\DS 1+2\mathcal R(\mu)(z-1) +(z-1)^2|\mu|^2=\left(1-2\mathcal R(\mu)+|\mu|^2\right)+2\left(\mathcal R(\mu)-|\mu|^2\right)z+|\mu|^2 z^2$. Since $\mu$ is not real, $\mu\neq 0$, yielding that $k=2$ and $1-2\mathcal R(\mu)+|\mu|^2=0$ and $\mathcal R(\mu)-|\mu|^2=0$. Since $1-2\mathcal R(\mu)+|\mu|^2=(1-\mathcal R(\mu))^2+\mathcal I(\mu)^2$, this yields that $\mathcal I(\mu)=0$, which contradicts that $\mu$ is not real. Hence, all eigenvalues $\mu$ of $K$ must be real and, again by Lemma \ref{LemmaSupMat}, for all eigenvalues $\mu$ of $K$, $1-\mu+z\mu$ must be a monomial, i.e., $\mu\in \{0,1\}$. 

Conversely, if $0$ and $1$ are the only eigenvalues of $K$, it is straightforward to see that Lemma \ref{LemmaSupMat} yields that $\E[z^{|Y|}]=z^p$, where $p$ is the multiplicity of the eigenvalue $1$ in $K$, yielding that $|Y|=p$ almost surely.		
\end{proof}

\subsection{Special classes of DPP's}

\subsubsection{Symmetric DPP's}

Most commonly, DPP's are defined with a real symmetric kernel $K$. In that case, it is well known (\cite{KulTas12}) that admissibility is equivalent to lie in the intersection $\mathcal S$ of two copies of the cone of positive semidefinite matrices: $K\succeq 0$ and $I-K\succeq 0$. DPP's with symmetric kernels possess a very strong property of negative dependence called \textit{negative association}. A simple observation is that if $Y\sim\textsf{DPP}(K)$ for some symmetric $K\in\mathcal S$, then $\cov(\mathds 1_{i\in Y},\mathds 1_{j\in Y})=-K_{i,j}^2\leq 0$, for all $i,j\in [N], i\neq j$. Moreover, if $J,J'$ are two disjoint subsets of $[N]$, then $\cov(\mathds 1_{J\subseteq Y},\mathds 1_{J'\subseteq Y})=\det(K_{J\cup J'})-\det(K_J)\det(K_J')\leq 0$. Negative association is the property that, more generally, $\cov(f(Y\cap J),g(Y\cap J))\leq 0$ for all disjoint subsets $J,J'\subseteq [N]$ and for all nondecreasing functions $f,g:\mathcal P([N])\to\R$ (i.e., $f(J_1)\leq f(J_2), \forall J_1\subseteq J_2\subseteq [N]$), where $\mathcal P([N])$ is the power set of $[N]$. We refer to \cite{borcea2009negative} for more details on the account of negative association. For their computational appeal, it is very tempting to apply DPP's in order to model interactions, e.g., as an alternative to Ising models. However, the negative association property of DPP's with symmetric kernels is unreasonably restrictive in several contexts, for it forces repulsive interactions between items. Next, we extend the class of DPP's with symmetric kernels in a simple way which is yet also allowing for attractive interactions.

\subsubsection{Signed DPP's}

We introduce the class $\mathcal T$ of \textit{signed kernels}, i.e., matrices $K\in\R^{N\times N}$ such that for all $i,j\in [N]$ with $i\neq j$, $K_{j,i}=\pm K_{i,j}$, i.e., $K_{j,i}=\epsilon_{i,j}K_{i,j}$ for some $\epsilon_{i,j}=\epsilon_{j,i}\in\{-1,1\}$. We call a \textit{signed DPP} any DPP with kernel $K\in\mathcal T$. In particular, if $K\in\mathcal T$ is an admissible kernel and $Y\sim\DPP(K)$, then for all $i,j\in[N]$ with $i\neq j$, $\textsf{cov}(\mathds 1_{i\in Y},\mathds 1_{j\in Y})=-\epsilon_{i,j}K_{i,j}^2$, which is of the sign of $-\epsilon_{i,j}$. In particular, when $\epsilon_{i,j}=-1$, this covariance is nonnegative, which breaks the negative association property of symmetric kernels.

\subsubsection{Signed block DPP's} 

As of particular interest, one can also consider block signed DPP's, with kernels $K\in\mathcal T$, where there is a partition of $[N]$ into pairwise disjoint, nonempty groups such that $K_{j,i}=-K_{i,j}$ if $i$ and $j$ are in the same group (hence, $i$ and $j$ \textit{attract} each other), $K_{j,i}=K_{i,j}$ if $i$ and $j$ are in different groups (hence, $i$ and $j$ \textit{repel} each other). 
As particular cases of signed block DPP's, consider those with block diagonal kernels $K$, where each block is skew-symmetric. It is easy to see that such DPP's can be written as the union of disjoint and independent DPP's, each corresponding to a diagonal block of $K$.

\subsection{Learning DPP's}

The main purpose of this work is to understand how to learn the kernel of a nonsymmetric DPP, given i.i.d. copies of that DPP. Namely, if $Y_1,\ldots,Y_n\stackrel{{\tiny \mbox{i.i.d.}}}{\sim}\DPP(K)$ for some unknown $K\in\TT$, how to estimate $K$ from the observation of $Y_1,\ldots,Y_n$? First comes the question of identifiability of $K$: two matrices $K,K'\in\TT$ can give rise to the same DPP. To be more specific, $\DPP(K)=\DPP(K')$ if and only if $K$ and $K'$ have the same list of principal minors. Hence, the kernel of a DPP is not necessarily unique. It is actually easy to see that it is unique if and only if it is diagonal. A first natural question that arises in learning the kernel of a DPP is the following:
\begin{center}
``\textit{What is the collection of all matrices $K\in\TT$ that produce a given DPP?}"
\end{center}
Given that the kernel of $Y_1$ is not uniquely defined, the goal is no longer to estimate $K$ exactly, but one possible kernel that would give rise to the same DPP as $K$. The route that we follow is similar to that followed by \cite{BruMoiRigUrs2017bis}, which is based on a method of moments. However, lack of symmetry of $K$ requires significantly different ideas. The idea is based on the fact that only few principal minors of $K$ are necessary in order to completely recover $K$ up to identifiability. Moreover, each principal minor $\Delta_J:=\det(K_J)$ can be estimated from the samples by $\hat\Delta_J=n^{-1}\sum_{i=1}^n \mathds 1_{J\subseteq Y_i}$. Since this last step is straightforward, we only focus on the problem of complete recovery of $K$, up to identifiability, given a list of few of its principal minors. In other words, we will ask the following question:
\begin{center}
``\textit{Given an available list of prescribed principal minors, how to recover a matrix $K\in\TT$ whose principal minors are given by that list, using as few queries from that list as possible?}"
\end{center}
This question, together with the one we asked for identifiability, is known as the \textit{principal minor assignment problem}, which we state precisely in the next section.

\subsection{The principal minor assignment problem}

The principal minor assignment problem (PMA) is a well known problem in linear algebra that consists of finding a matrix with a prescribed list of principal minors \cite{Rising2015}. Let $\mathcal H\subseteq\C^{N\times N}$ be a collection of matrices. Typically, $\mathcal H$ is the set of Hermitian matrices, or real symmetric matrices or, in this work, $\mathcal H=\mathcal T$. Given a list $(a_J)_{J\subseteq [N], J\neq\emptyset}$ of $2^N-1$ complex numbers, (PMA) asks the following two questions:
\begin{enumerate}[label=(PMA\arabic*)]
	\item \label{PMA1} Find a matrix $K\in\HH$ such that $\det(K_J)=a_J$, $\forall J\subseteq [N], J\neq\emptyset$.
	\item \label{PMA2} Describe the set of all solutions of \ref{PMA1}. 
\end{enumerate}

A third question, which we do not address here, is to decide whether \ref{PMA1} has a solution. It is known that this would require the $a_J$'s to satisfy polynomial equations \cite{Oeding2011}. Here, we assume that a solution exists, i.e., the list $(a_J)_{J\subseteq [N], J\neq\emptyset}$ is a valid list of prescribed principal minors, and we aim to answer \ref{PMA1} efficiently, i.e., output a solution in polynomial time in the size $N$ of the problem, and to answer \ref{PMA2} at a purely theoretical level. In the framework of DPP's, \ref{PMA1} is related to the problem of estimating $K$ by a method of moments and \ref{PMA2} concerns the identifiability of $K$, since the set of all solutions of \ref{PMA1} is the identifiable set of $K$.

\section{Solving the principal minor assignment problem for nonsymmetric DPP's} \label{Sec:Sol}

	\subsection{Preliminaries: PMA for symmetric matrices}

Here, we briefly describe the PMA problem for symmetric matrices, i.e., $\mathcal H=\SN$, the set of real symmetric $N\times N$ matrices. This will give some intuition for the next section.

\begin{fact}\label{Fact1} 
	The principal minors of order one and two of a symmetric matrix completely determine its diagonal entries and the magnitudes of its off diagonal entries. 
\end{fact}

The adjacency graph $G_K=([N],E_K)$ of a matrix a matrix $K\in\SN$ is the undirected graph on $N$ vertices, where, for all $i,j\in [N]$, $\{i,j\}\in E_K\iff K_{i,j}\neq 0$. As a consequence of Fact \ref{Fact1}, we have:

\begin{fact}\label{fact2}
	The adjacency graph of any symmetric solution of \ref{PMA1} can be learned by querying the principal minors of order one and two. Moreover, any two symmetric solutions of \ref{PMA1} have the same adjacency graph.
\end{fact}

Then, the signs of the off diagonal entries of a symmetric solution of \ref{PMA1} should be determined using queries of higher order principal minors, and the idea is based on the next fact. For a matrix $K\in\SN$ and a cycle $C$ in $G_K$, denote by $\pi_K(C)$ the product of entries of $K$ along the cycle $C$, i.e., $\DS \pi_K(C)=\prod_{\{i,j\}\in C:i<j}K_{i,j}$.

\begin{fact}\label{fact3}
	For all matrices $K\in\SN$ and all $J\subseteq [N]$, $\det(K_J)$ only depends on the diagonal entries of $K_J$, the magnitude of its off diagonal entries and the $\pi_K(C)$, for all cycles $C$ in the subgraph of $G_K$ where all vertices $j\notin J$ have been deleted. 
\end{fact}

Fact \ref{fact3} is a simple consequence of the fundamental formula:
\begin{equation} \label{eq:det}
	\det(K_J)=\sum_{\sigma\in\Sig_J}(-1)^\sigma\prod_{j\in J}K_{j,\sigma(j)},
\end{equation}
where $\mathcal\Sig_J$ is the group of permutations of $J$, Moreover, every permutation $\sigma\in\Sig$ can be decomposed as a product of cyclic permutations. Finally, every undirected graph has a cycle basis made of induced cycles, i.e., there is a \textit{small} family $\mathcal B$ of induced cycles such that every cycle (seen as a collection of edges) in the graph can be decomposed as the symmetric difference of cycles that belong to $\mathcal B$. Then, it is easy to see that for all cycles $C$ in the graph $G_K$, $\pi_K(C)$ can be written as the product of some $\pi_K(\tilde C)$, for some cycles $\tilde C\in\mathcal B$ and of some $K_{i,j}^2$'s, $i\neq j$. Moreover, for all induced cycles $C$ in $G_K$, $\pi_K(C)$ can be determined from $\det(K_J)$, where $J$ is the set of vertices of $C$. Since, by Fact \ref{fact2}, $G_K$ can be learned, what remains is to find a cycle basis of $G_K$, made of induced cycles only, which can be performed in polynomial time (see \cite{horton1987polynomial,Amaldi2010}) and, for each cycle $C$ in that basis, query the corresponding principal minor of $K$ in order to learn $\pi_K(C)$. Finally, in order to determine the signs of the off diagonal entries of $K$, find a sign assignment that matches with the signs of the $\pi_K(C)$, for $C$ in the aforementioned basis. Finding such a sign assignment consists of solving a linear system in $\GF$ (see Section 1 in the Supplementary Material).

%
%

\subsection{PMA when $\mathcal H=\mathcal T$, general case} \label{sec:gen}
	
We now turn to the case $\mathcal H=\mathcal T$. First, as in the symmetric case, the diagonal entries of any matrix $K\in\TT$ are given by its principal minors of order 1. Now, let $i<j$ and consider the principal minor of $K$ corresponding to $J=\{i,j\}$:
\begin{equation*}
	\det(K_{\{i,j\}})=K_{i,i}K_{j,j}-\epsilon_{i,j}K_{i,j}^2.
\end{equation*}
Hence, $|K_{i,j}|$ and $\epsilon_{i,j}$ can be learned from the principal minors of $K$ corresponding to the sets $\{i\}, \{j\}$ and $\{i,j\}$.

Note that if $K\in\TT$, one can still define its adjacency graph $G_K$ as in the symmetric case, since $K_{i,j}\neq 0\iff K_{j,i}\neq 0$, for all $i\neq j$. Recall that we identify a cycle of a graph with its edge set. For all $K\in\TT$ and for all cycles $C$ in $G_K$, let $\DS \epsilon_K(C)=\prod_{\{i,j\}\in C:i<j}\epsilon_{i,j}$ be the product of the $\epsilon_{i,j}$'s along the edges of $C$, where $\epsilon_{i,j}\in\{-1,1\}$ is such that $K_{i,j}=\epsilon_{i,j}K_{j,i}$. Note that the condition ``$i<j$" in the definition of $\epsilon_K(C)$ is only to ensure no repetition in the product. Now, unlike in the symmetric case, we need to be more careful when defining $\pi_K(C)$, for a cycle $C$ of $G_K$, since the direction in which $C$ is traveled matters. 

\begin{definition}
	A signed graph is an undirected graph $([N],E)$ where each edge is assigned a sign $-1$ or $+1$. 
\end{definition}

In the sequel, we make the adjacency graph $G_K$ of any matrix $K\in \TT$ signed by assigning $\epsilon_{i,j}$ to each edge $\{i,j\}$ of the graph. As we noticed above, the signed adjacency graph of $K$ can be learned from its principal minors of orders one and two. Unlike in the symmetric case, induced cycles might be of no help to determine the signs of the off diagonal entries of $K$.

\begin{wrapfigure}{R}{0.3\textwidth}
\centering
\includegraphics[width=0.2\textwidth]{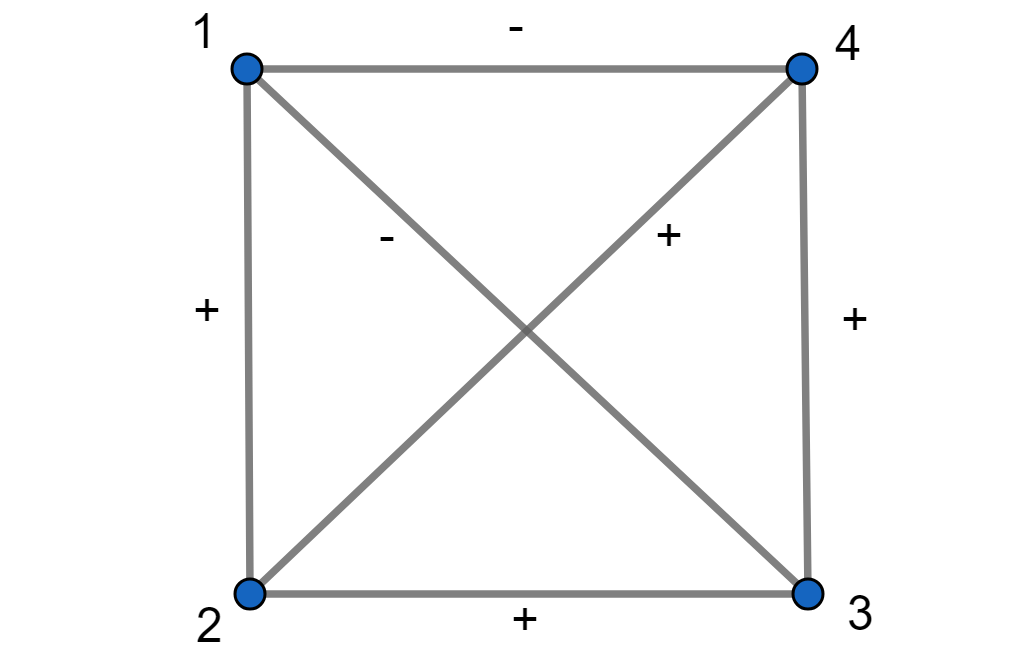}
\caption{\label{fig:fig1}}{A signed graph}
\end{wrapfigure}

\begin{definition}
	Let $G$ be an undirected graph and $C$ a cycle of $G$. A \textit{traveling} of $C$ is an oriented cycle of $G$ whose vertex set coincides with that of $C$. The set of travelings of $C$ is denoted by $\mathbb T(C)$.
\end{definition}

For instance, an induced cycle has exactly two travelings, corresponding to the two possible orientations of $C$. 

In Figure \ref{fig:fig1}, the cycle $C=1\leftrightarrow 2 \leftrightarrow 3\leftrightarrow 4\leftrightarrow 1$ has six travelings: $\overrightarrow{C_1}=1\to 2\to 3\to 4\to 1$, $\overrightarrow{C_2}=1\to 4\to 3\to 2\to 1$, $\overrightarrow{C_3}=1\to 2\to 4\to 3\to 1$, $\overrightarrow{C_4}=1\to 3\to 4\to 2\to 1$, $\overrightarrow{C_5}=1\to 4\to 2\to 3\to 1$ and $\overrightarrow{C_6}=1\to 3\to 2\to 4\to 1$.

Formally, while we identify a cycle with its edge set (e.g., $C=\{\{1,2\},\{2,3\},\{3,4\},\{1,4\}\}$, we identify its travelings with sets of ordered pairs corresponding to their oriented edges (e.g., $\overrightarrow{C_1}=\{(1,2),(2,3),(3,4),(4,1)\}$). Also, for simplicity, we always denote oriented cycles using the symbol $\overrightarrow{\cdot}$ (e.g., $\overrightarrow{C}$ as opposed to $C$, which would stand for an unoriented cycle).
\begin{definition}
	Let $K\in\TT$ and $C$ be a cycle in $G_K$. We denote by $\DS \pi_K(C)=\sum_{\overrightarrow{C}\in\mathbb T(C)}\prod_{(i,j)\in \overrightarrow{C}}K_{i,j}$.
\end{definition}

For example, if the graph in Figure \ref{fig:fig1} is the adjacency graph of some $K\in\TT$ and $C$ is the cycle $C=1\leftrightarrow 2 \leftrightarrow 3\leftrightarrow 4\leftrightarrow 1$, then,
\begin{align*}
	\pi_K(C) & = K_{1,2}K_{2,3}K_{3,4}K_{4,1}+K_{1,4}K_{4,3}K_{3,2}K_{2,1}+K_{1,2}K_{2,4}K_{4,3}K_{3,1}+K_{1,3}K_{3,4}K_{4,2}K_{2,1} \\\
	& \hspace{8mm} +K_{1,4}K_{4,2}K_{2,3}K_{3,1}+K_{1,3}K_{3,2}K_{2,4}K_{4,1} \\
	& = \left(1+\epsilon_K(\overrightarrow{C_1})\right) K_{1,2}K_{2,3}K_{3,4}K_{4,1}+\left(1+\epsilon_K(\overrightarrow{C_3})\right)K_{1,2}K_{2,4}K_{4,3}K_{3,1} \\
	& \hspace{8mm} +\left(1+\epsilon_K(\overrightarrow{C_5})\right)K_{1,4}K_{4,2}K_{2,3}K_{3,1} \\
	& = 2 K_{1,3}K_{3,2}K_{2,4}K_{4,1}.
\end{align*}
where the oriented cycles $\overrightarrow{C_1}$, $\overrightarrow{C_3}$ and $\overrightarrow{C_5}$ are given above, and where we use the shortcut $\epsilon_K(\overrightarrow{C_j})$ ($j=1,3,5$) to denote $\epsilon_K(C_j)$, where $C_j$ is the unoriented version of $\overrightarrow{C_j}$. 

In the same example, there are only two triangles $T$ (i.e., cycles of size $3$) that satisfy $\pi_K(T)\neq 0$: $1\leftrightarrow 3 \leftrightarrow 4\leftrightarrow 1$ and $2 \leftrightarrow 3\leftrightarrow 4\leftrightarrow 2$.

The following result, yet a simple consequence of \eqref{eq:det}, is fundamental.
\begin{lemma} \label{lem:fundam}
For all $J\subseteq [N]$, $\det(K_J)$ can be written as a function of the $K_{i,i}, K_{i,j}^2, \epsilon_{i,j}$'s, for $i,j\in J, i\neq j$ and $\pi_K(C)$'s, for all cycles $C$ in $G_{K_J}$, the subgraph of $G_K$ where all vertices $j\notin J$ are removed.
\end{lemma} 

\begin{proof}
Write a permutation $\sigma\in\mathfrak S_J$ as a product of cyclic permutations $\sigma=\sigma_1\circ\sigma_2\circ\ldots\circ\sigma_p$. For each $j=1,\ldots,p$, assume that $\sigma_j$ correspond to an oriented cycle $\overrightarrow{C_j}$ of $G_K$, otherwise the contribution of $\sigma$ to the sum \eqref{eq:det} is zero. Then, the lemma follows by grouping all permutations in the sum \eqref{eq:det} that can be decomposed as a product of $p$ cyclic permutations $\sigma_1',\ldots,\sigma_p'$ where, for all $j=1,\ldots,p$, $\sigma_j'$ has the same support as $\sigma_j$.
\end{proof}

As a consequence, we note that unlike in the symmetric case, the signs of the off diagonal entries can no longer be determined using a cycle basis of induced cycles, since such a basis may contain only cycles which have no contribution to the principal minors of $K$. In the same example as above, the only induced cycles of $G_K$ are triangles, and any cycle basis should contain at least three cycles. However, there are only four triangles in that graph and two of them have a zero contribution to the principal minors of $K$. Hence, in that case, it is necessary to query principal minors that do not correspond to induced cycles in order to find a solution to \ref{PMA1}.

In order to summarize, we state the following theorem. 

\begin{theorem} \label{thm:gen}
	Let $H,K\in\TT$. The following statements are equivalent.
	\begin{itemize}
		\item $H$ and $K$ have the same list of principal minors.
		\item $H_{i,i}=K_{i,i}$ and $|H_{i,j}|=|K_{i,j}|$, for all $i,j\in [N]$ with $i\neq j$, $H$ and $K$ have the same signed adjacency graph and, for all cycles $C$ in that graph, $\pi_K(C)=\pi_H(C)$.
	\end{itemize}
\end{theorem}

Theorem \ref{thm:gen} does not provide any insight on how to solve \ref{PMA2} efficiently, since the number of cycles in a graph can be exponentially large in the size of the graph. A refinement of this theorem, where we would characterize a minimal set of cycles, that could be found efficiently and that would characterize the principal minors of $K\in\mathcal T$ (such as a basis of induced cycles, in the symmetric case), is an open problem. However, in the next section, we refine this result for a smaller class of nonsymmetric kernels.

\subsection{PMA when $\mathcal H=\TT$, dense case}

In this section, we only consider matrices $K\in\TT$ such that for all $i,j\in [N]$ with $i\neq j$, $K_{i,j}\neq 0$. The adjacency graph of such a matrix is a signed version of the complete graph, which we denote by $G_N$. We also assume that for all pairwise distinct $i,j,k,l\in [N]$ and all $\eta_1,\eta_2,\eta_3\in\{-1,0,1\}$, 
\begin{equation} \label{eq:cond}
	\eta_1 K_{i,j}K_{j,k}K_{k,l}K_{l,i}+\eta_2 K_{i,j}K_{j,l}K_{l,k}K_{k,i}+\eta_3 K_{i,k}K_{k,j}K_{j,l}K_{l,i}=0\Rightarrow \eta_1=\eta_2=\eta_3=0.
\end{equation}
Note that Condition \eqref{eq:cond} only depends on the magnitudes of the entries of $K$. Hence, if one solution of \ref{PMA1} satisfies \eqref{eq:cond}, then all the solutions must satisfy it too. Condition \eqref{eq:cond} is not a strong condition: Indeed, any generic matrix with rank at least $4$ is very likely to satisfy it.

For the sake of simplicity, we restate \ref{PMA1} and \ref{PMA2} in the following way. Let $K\in\TT$ be a ground kernel satisfying the two conditions above (i.e., $K$ is dense and satisfies Condition \ref{eq:cond}), and assume that $K$ is unknown, but its principal minors are available.

\begin{enumerate}[label=(PMA'\arabic*)]
	\item \label{PMA'1} Find a matrix $H\in\TT$ such that $\det(H_J)=\det(K_J)$, $\forall J\subseteq [N], J\neq\emptyset$.
	\item \label{PMA'2} Describe the set of all solutions of \ref{PMA'1}. 
\end{enumerate}
Moreover, recall that we would like to find a solution to \ref{PMA'1} that uses few queries from the available list of principal minors of $K$, in order to design an algorithm that is not too costly computationally.

Since $K$ is assumed to be dense, every subset $J\subseteq [N]$ of size at least 3 is the vertex set of a cycle. Moreover, for all cycles $C$ of $G_N$, $\pi_K(C)$ only depends on the vertex set of $C$, not its edge set. Therefore, in the sequel, for the ease of notation, we denote by $\pi_K(J)=\pi_K(C)$ for any cycle $C$ with vertex set $J$.

The main result of this section is stated in the following theorem.

\begin{theorem} \label{thm:dense}
	A matrix $H\in\TT$ is a solution of \ref{PMA'1} if and only if it satisfies the following requirements:
\begin{itemize}
	\item $H_{i,i}=K_{i,i}$ and $|H_{i,j}|=|K_{i,j}|$, for all $i,j\in [N]$ with $i\neq j$;
	\item $H$ has the same signed adjacency graph as $K$, i.e., $G_H=G_K=G_N$ and $\DS \frac{H_{i,j}}{H_{j,i}}=\frac{K_{i,j}}{K_{j,i}}$, for all $i\neq j$;
	\item $\pi_H(J)=\pi_K(J)$, for all $J\subseteq [N]$ of size $3$ or $4$.  
\end{itemize}	 
\end{theorem}

The left to right implication follows directly from Theorem 1. Now, let $H$ satisfy the four requirements, and let us prove that 
\begin{equation} \label{eq:step}
	\det(H_J)=\det(K_J), 
\end{equation}
for all $J\subseteq [N]$. If $J$ has size 1 or 2, \eqref{eq:step} is straightforward, by the first three requirements. If $J$ has size 3 or 4, it is easy to see that $\det(H_J)$ only depends on $H_{i,i}$, $H_{i,j}^2, i,j\in J$ and $\pi_H(S), S\subseteq J$, hence, \eqref{eq:step} is also granted. Now, let $J\subseteq [N]$ have size at least 5. By Lemma 1, it is enough to check that 
\begin{equation} \label{eq:step11}
\pi_H(S)=\pi_K(S),
\end{equation}
for all $S\subseteq J$ of size at least 3. 

Let us introduce some new notation for the rest of the proof. For all oriented cycles $\overrightarrow{C}$ in $G_N$, we denote by $\overrightarrow{\pi}_K(\overrightarrow{C})=\prod_{(i,j)\in \overrightarrow{C}}K_{i,j}$ and $\overrightarrow{\pi}_H(\overrightarrow{C})=\prod_{(i,j)\in \overrightarrow{C}}H_{i,j}$. Let $J\subseteq [N]$ of size at least $3$. In the sequel, for each unoriented cycle $C$ with vertex set $J$, let $\overrightarrow{C}$ be any of the two possible orientations of $C$, chosen arbitrarily. Denote by $\mathbb T^+(J)$ the set of unoriented cycles $C$ with vertex set $J$, such that $\epsilon_K(C)=+1$. It is clear that
\begin{equation} \label{eq:step2}
	\pi_H(J)=2\sum_{C\in\mathbb T^+(J)}\overrightarrow{\pi}_H(\overrightarrow{C}),
\end{equation}
and the same holds for $K$.
Now, let $\mathcal J^+=\{(i,j,k)\subseteq [N]:i\neq j, i\neq k, j\neq k, \epsilon_{i,j}\epsilon_{j,k}\epsilon_{i,k}=+1\}$ be the set of \textit{positive triangles}, i.e., the set of triples that define triangles in $G_N$ that do contribute to the principal minors of $K$. 
The requirements on $H$ ensure that $H_{i,j}H_{j,k}H_{i,k}=K_{i,j}K_{j,k}K_{i,k}$ for all $(i,j,k)\in \mathcal J^+$ and, by Condition (3), using \eqref{eq:step2}, that $\overrightarrow{\pi}_H(\overrightarrow{C})=\overrightarrow{\pi}_K(\overrightarrow{C})$, for all cycles $C$ of length $4$ with $\epsilon_K(C)=1$ (where, we recall that $C$ is the unoriented version of the oriented cycle $\overrightarrow{C}$).

Let $p$ be the size of $S$. By \eqref{eq:step2}, it is enough to check that $\overrightarrow{\pi}_H(\overrightarrow{C})=\overrightarrow{\pi}_K(\overrightarrow{C})$ for all positive oriented cycles $\overrightarrow{C}$ of length $p$, i.e., for all oriented cycles $\overrightarrow{C}$ of length $p$ with $\epsilon_K(C)=+1$. Let us prove this statement by induction on $p$. If $p=3$ or 4, \eqref{eq:step11} is granted by the requirement imposed on $H$. Let $p=5$.  Let $\overrightarrow{C}$ be a positive oriented cycle of length $5$. Without loss of generality, let us assume that $\overrightarrow{C}=1\to 2\to 3\to 4\to 5\to 1$. Since it is positive, it can have either 0, 2 or 4 negative edges. Suppose it has 0 negative edges, i.e., all its edges are positive (i.e., satisfy $\epsilon_{i,j}=+1$). We call a chord of the cycle $C$ any edge between two vertices of $C$, that is not an edge in $C$. If $C$ has a positive chord, i.e., if there are two vertices $i\neq j$ with $j\neq i\pm 1 \hspace{2mm}(\textsf{mod} 5)$ and $\epsilon_{i,j}=+1$, then $C$ can be decomposed as the symmetric difference of two positive cycles $C_1$ and $C_2$, one of length 3, one of length 4, with $\DS \overrightarrow{\pi}_H(\overrightarrow{C})=\frac{\overrightarrow{\pi}_H(\overrightarrow{C_1})\overrightarrow{\pi}_H(\overrightarrow{C_2})}{H_{i,j}^2}=\frac{\overrightarrow{\pi}_K(\overrightarrow{C_1})\overrightarrow{\pi}_K(\overrightarrow{C_2})}{K_{i,j}^2}=\overrightarrow{\pi}_K(\overrightarrow{C})$. Now, assume that all chords of $C$ are negative. Then, the cycles $\overrightarrow{C_1}=1\to 2\to 4\to 3\to 1$, $\overrightarrow{C_2}=1\to 3\to 5\to 1$ and $\overrightarrow{C_3}=2\to 4\to 5 \to 3\to 2$ are positive and it is easy to see that $\DS \overrightarrow{\pi}_H(\overrightarrow{C})=\frac{\overrightarrow{\pi}_H(\overrightarrow{C_1})\overrightarrow{\pi}_H(\overrightarrow{C_2})\overrightarrow{\pi}_H(\overrightarrow{C_3})}{H_{1,3}^2 H_{2,4}^2 H_{3,5}^2}=\frac{\overrightarrow{\pi}_K(\overrightarrow{C_1})\overrightarrow{\pi}_K(\overrightarrow{C_2})\overrightarrow{\pi}_K(\overrightarrow{C_3})}{K_{1,3}^2 K_{2,4}^2 K_{3,5}^2}=\overrightarrow{\pi}_K(\overrightarrow{C})$, where we only used the requirements imposed on $H$. The cases when 
$\overrightarrow{C}$ has two or four negative edges are treated similarly and they are skipped here. 

Let $p\geq 6$ and, without loss of generality, let us assume that $\overrightarrow{C}=1\to 2\to \ldots \to p-1\to p$. If $\overrightarrow{C}$ has a chord $(i,j)$ that splits $\overrightarrow{C}$ into two positive cycles $\overrightarrow{C_1}$ and $\overrightarrow{C_2}$, then as above, we write $\DS \overrightarrow{\pi}_H(\overrightarrow{C})=\frac{\overrightarrow{\pi}_H(\overrightarrow{C_1})\overrightarrow{\pi}_H(\overrightarrow{C_2})}{\epsilon_{i,j}H_{i,j}^2}=\frac{\overrightarrow{\pi}_K(\overrightarrow{C_1})\overrightarrow{\pi}_K(\overrightarrow{C_2})}{\epsilon_{i,j}K_{i,j}^2}=\overrightarrow{\pi}_K(\overrightarrow{C})$, where we use the induction. Otherwise, assume that there is no chord that splits $\overrightarrow{C}$ into two positive cycles. In that case, the three cycles $\overrightarrow{C_1}=1\to 2 \to 3 \to 5 \to 1$, $\overrightarrow{C_2}=1\to 3 \to 4 \to 5 \to 1$, and $\overrightarrow{C_3}=1 \to 3 \to 5 \to 6 \to 7 \to 8 \to \ldots\to p \to 1$ must be positive, and we have $\DS \overrightarrow{\pi}_H(\overrightarrow{C})=\frac{\overrightarrow{\pi}_H(\overrightarrow{C_1})\overrightarrow{\pi}_H(\overrightarrow{C_2})\overrightarrow{\pi}_H(\overrightarrow{C_3})}{K_{1,3}^2K_{3,5}^2K_{1,5}^2}=\frac{\overrightarrow{\pi}_K(\overrightarrow{C_1})\overrightarrow{\pi}_K(\overrightarrow{C_2})\overrightarrow{\pi}_K(\overrightarrow{C_3})}{K_{1,3}^2K_{3,5}^2K_{1,5}^2}=\overrightarrow{\pi}_K(\overrightarrow{C})$, by induction.

Finally, we provide an algorithm that finds a solution to \ref{PMA'1} in polynomial time.


\begin{theorem}
	Algorithm \ref{alg:mainalg} finds a solution of \ref{PMA'1} in polynomial time in $N$.
\end{theorem}

\begin{proof}
	The fact that Algorithm \ref{alg:mainalg} finds a solution of \ref{PMA'1} is a straightforward consequence of Theorem \ref{thm:dense}. Its complexity is of the order of that of Gaussian elimination for a linear system of at most $O(N^4)$ equations, corresponding to cycles of size at most 4 and with $O(N^2)$ variables, corresponding to the entries of $H$.
\end{proof}

\begin{algorithm}[tb]
   \caption{Find a solution $H$ to \ref{PMA'1}}
   \label{alg:mainalg}
\begin{algorithmic}
   \STATE {\bfseries Input:} List $\{a_J:J\subseteq [N]\}$.
      \vspace{0.5pc}
     \STATE Set $H_{i,i}=a_{\{i\}}$ for all $i=1,\ldots,N$.
     \STATE Set $|H_{i,j}| = \left|a_{\{i\}}a_{\{j\}}-a_{\{i,j\}}\right|$ for all $i\neq j$.
     \STATE Set $\DS \epsilon_{i,j}=\textsf{sign}\left(a_{\{i\}}a_{\{j\}}-a_{\{i,j\}}\right)$ for all $i\neq j$.
     \STATE Find the set $\mathcal J^+$ of all triples $(i,j,j)$ of pairwise distinct indices such that $\epsilon_{i,j}\epsilon_{i,k}\epsilon_{j,k}=1$ and find the sign of $H_{i,j}H_{j,k}H_{i,k}$ for all $(i,j,k)\in\mathcal J^+$, using $a_J, J\subseteq {i,j,k}$.
     \STATE For all $S\subseteq [N]$ of size $4$, find $\pi_H(S)$ and deduce the sign of $\overrightarrow{\pi}_K(\overrightarrow{C})$, for all $\overrightarrow{C}\in\mathbb T^+(S)$
     \STATE Find an sign assignment for the off diagonal entries of $H$ that matches all the signs found in the previous step, by Gaussian elimination in $\GF$.
     \end{algorithmic}
\end{algorithm}

\section{Conclusion}

The main goal of this work was to study DPP's with nonsymmetric kernels. As a particular case, we have introduced signed DPP's, which allow for both repulsive and attractive interactions. By solving the PMA problem, we have characterized identification of the kernel in the dense case (Theorem \ref{thm:dense}) and we have given an algorithm that finds a dense matrix $H\in \TT$ with prescribed principal minors, in polynomial time in the size $N$ of the unknown matrix. In practice, these principal minors are unknown, but they can be estimated from observed samples from a DPP. As long as the adjacency graph can be recovered exactly from the samples, which would be granted with high probability for a large number of observations, and if all entries of $H$ are bounded away from zero by some known constant (that depends on $N$), solving the PMA problem amounts in finding the signs of the entries of $H$, up to identifiability, which can also be done exactly with high probability, if the number of observed samples is large (see, e.g., \cite{BruMoiRigUrs2017bis}). 

However, extending classical symmetric DPP's to non symmetric kernels poses some questions: For instance, we do not know how to sample a signed DPP efficiently. For symmetric kernels, sampling a DPP can be done using a spectral decomposition of the kernel \cite[Section 2.4.4]{KulTas12} or a Markov chain Monte Carlo method, relying on the fact that $\DPP(K)$ is a strongly Rayleigh distribution for symmetric $K$ \cite{anari2016monte}. The latter property no longer holds for general DPP's, and the role of the eigenstructure of a general admissible kernel $K$ is not clear yet, since the eigenvalues may not be real.

\bibliographystyle{plain}
\bibliography{Biblio}

\appendix 
\section{Appendix: Gaussian elimination for sign systems}

The last step of Algorithm \ref{alg:mainalg} requires to solve a system of equations with variables in $\{-1,+1\}$. These equations are of the form $\prod_{e\in C}x_e=b_C, C\in\mathcal B$, where $x_1,\ldots,x_m\in \{-1,+1\}$ are the unknown variables, $C\subseteq [m]$ and $b_C\in \{-1,1\}$. The sign space $\{-1,+1\}$ can be equipped with a linear structure; Then, the initial system becomes linear and it can be solved using traditional Gaussian elimination.

\subsection{Equipping signs with a linear structure}

First, equip $\Signs$ with its canonical multiplication in order to make it an Abelian group, with $+1$ as its neutral element (hence, $+1$ will play the role of the null vector once $\Signs$ is equipped with a linear structure). In common linear spaces, this operation is usually denoted as an addition. Note that here, the multiplication plays both the role of the usual addition, and that of the usual subtraction.

 Then, we define an operation on $\GF\times \{-1,1\}$ as follows: For all $x\in\Signs$, $0.x=+1$ and $1.x=x$. It is easy to see that this defines a linear structure on $\Signs$ over the field $\GF$ and this linear space has dimension 1. It follows that for each positive integer $m$, the space $\Signs^m$ is also canonically equipped with a linear structure over the field $\GF$, and it has dimension $m$. A natural basis for this linear space is given by $e_1,\ldots,e_m$, where $e_i$ is the vector in $\Signs^m$ with $i$-th coordinate $-1$ and all other coordinates $+1$. Then, any vector $x=(x_1,\ldots,x_m)\in\Signs^m$ can be decomposed as $\prod_{i=1}^m \tilde x_i.e_i$, where $\tilde x_i=0$ if $x_i=+1$, $\tilde x_i=1$ if $x_i=-1$ ($\tilde x_i\in\GF$). \textit{Note that here, the product plays the role of the usual sum on linear spaces, and must be computed coordinatewise.} Hence, the vector $x\in\Signs^m$ can be represented as an $m$-dimensional vector $\tilde x=(\tilde x_1,\ldots,\tilde x_m)^{\top}$ in $\GF^m$, and this defines an isomorphism between the linear spaces $\Signs^m$ and $\GF^m$, where $\GF^m$ is equipped with the regular addition. In particular, if $\lambda,\mu\in\GF$, $x=\prod_{i=1}^m \tilde x_i.e_i\in\Signs^m$ and $y=\prod_{i=1}^m \tilde y_i.e_i\in\Signs^m$, then $(\lambda.x)(\mu.y)=\prod_{i=1}^m (\lambda\tilde x_i+\mu\tilde y_i).e_i$. As a consequence, solving a linear system in $\Signs^m$ amounts to solving a linear system in $\GF$.

\subsection{Linear systems on signs and Gaussian elimination}

Let $m,p$ be positive integers, $C_1,\ldots,C_p$ be subsets of $[m]$ and $b_1,\ldots,b_p\in\Signs$. Consider the system of equations $\prod_{i\in C_k}x_i=b_k, k=1,\ldots,p$, with unknown variables $x_1,\ldots,x_m\in \Signs$. For $i\in [m]$, define $\tilde x_i\in\GF$ as above: $\tilde x_i=0$ if $x_i=+1$ and $\tilde x_i=1$ if $x_i=-1$. Define $\tilde b_k$ similarly for $k\in [p]$. Then, the linear system is equivalent to $\sum_{i\in C_k}\tilde x_i=\tilde b_k$, for all $k\in [p]$. A solution of this system, with unknown variables $\tilde x_1,\ldots,\tilde x_m\in\GF$, if any, can be found using standard Gaussian elimination, where all the sums must be understood modulo $2$. 

\end{document}